\documentclass[11pt]{amsart}
\usepackage{graphicx} 
\usepackage{amsthm}
\usepackage{times,amssymb,amsmath,amsfonts}
\usepackage{xfrac}

\usepackage{hyperref}
\usepackage{fix-cm}

\usepackage{color}
\usepackage{amsmath,amssymb,amsthm}
\usepackage{wrapfig}
\usepackage{ifpdf}

\newtheorem{thm}{Theorem}[section]

\newtheorem{definition}[thm]{Definition}

\newtheorem{remark}[thm]{Remark}

\title[$T$-avoiding spherical codes in $\mathbb{R}^{32}$]{On $T$-avoiding spherical codes and designs in $\mathbb{R}^{32}$}

\begin{document}

\author[P. Boyvalenkov]{P. Boyvalenkov}
\address{ Institute of Mathematics and Informatics, Bulgarian Academy of Sciences,
8 G Bonchev Str., 1113  Sofia, Bulgaria}
\email{peter@math.bas.bg}

\author[D. Cherkashin]{D. Cherkashin}
\address{ Institute of Mathematics and Informatics, Bulgarian Academy of Sciences,
8 G Bonchev Str., 1113  Sofia, Bulgaria}
\email{jiocb@math.bas.bg}

\author[P. Dragnev]{P. Dragnev}
\address{ Department of Mathematical Sciences, Purdue University 
Fort Wayne, Fort Wayne, IN 46805, USA }
\email{dragnevp@purdue.edu}

\author[D. Yorgov]{D. Yorgov}
\address{ Department of Mathematical Sciences, Purdue University 
Fort Wayne, Fort Wayne, IN 46805, USA }
\email{yorgovd@purdue.edu}

\author[V. Yorgov]{V. Yorgov}
\address{ Department of Mathematics and Computer Science, Fayetteville State University, Fayetteville, NC 28301, USA }
\email{vyorgov@gmail.com}

\begin{abstract}
In this article, we show that the minimal vectors of the extremal even unimodular lattices in $\mathbb{R}^{32}$ are $T$-avoiding universally optimal for suitable sets $T$. Moreover, they are minimal $T$-avoiding spherical designs and maximal $T$-avoiding codes for appropriate choices of $T$.
\end{abstract}

\maketitle

\section{Introduction}

Let $\mathbb{S}^{n-1}=\{\mathbf{x} = (x_1,\ldots,x_n): x_1^2+\cdots+x_n^2=1\}$ be the 
unit sphere in $\mathbb{R}^n$. A spherical code is a non-empty finite set $C \subset \mathbb{S}^{n-1}$. Consider  
\[ I (C):=\{ \mathbf{x \cdot y} : \mathbf{x}, \mathbf{y} \in C, \mathbf{x} \neq \mathbf{y}\}, \]
the set of all inner products of distinct points of $C$, and denote by 
$s = \max I(C)$ the largest inner product (also known as a maximal cosine) of $C$. We refer to $C$ as an $s$-code or as a spherical $(n,N,s)$-code, where $|C|=N$.

We are interested in spherical codes and designs (in particular, $s$-codes and $\tau$-designs) which avoid some sets of inner products in $[-1,1)$.

\begin{definition}
Let $T \subset [-1,1)$. A spherical code $C \subset \mathbb{S}^{n-1}$  is called $T$-avoiding if $I(C) \cap T = \varnothing$. If $C$ is an $s$-code, then we assume $T \subset [-1,s)$ to avoid trivial cases. 
\end{definition}

Spherical codes that are {\it well-distributed} have many applications in various fields. Of course, well-distributed can have many meanings (and so many applications). To quantify it, the notion of the discrete $h$-energy of a spherical code $C$ for a given interaction potential $h$ is considered (see, e.g.,~\cite{BHS2019,CS} and references therein).

\begin{definition}
Given a function $h:[-1,1]\to (-\infty,+\infty]$, continuous on $(-1,1)$, we consider the $h$-energy of a spherical code $C$ as 
\begin{equation} \label{h-energy} E_h (C):=\sum_{\mathbf{x} \neq \mathbf{y} \in C} h( \mathbf{x \cdot y}).
\end{equation} 
A spherical code $C$ is called \emph{universally optimal}~\cite{CK2007}, if it has minimum $h$-energy among all codes on $\mathbb{S}^{n-1}$ of cardinality $|C|$ for all absolutely monotone potentials\footnote{A potential function $h$ is absolutely monotone if its derivatives are nonnegative, i.e. $h^{(j)}\geq 0$, $j\geq 1.$ } $h$.
\end{definition}

The notion of $T$-avoiding universally optimal codes is then introduced in a natural way.

\begin{definition}
Given $T\subset [-1,1)$, a $T$-avoiding code $C$ is called \emph{$T$-avoiding universally optimal} if for any absolutely monotone potential $h$ it has minimum $h$-energy among all $T$-avoiding codes on $\mathbb{S}^{n-1}$ with cardinality $|C|$. 
\end{definition}

An important meaning of good distribution is given by the concept of spherical designs~\cite{DGS}. 

\begin{definition}
A spherical $\tau$-design is a spherical 
code $C \subset \mathbb{S}^{n-1}$ such that
\begin{equation} \label{MultiPolQ}
\int_{\mathbb{S}^{n-1}} p(\mathbf{x}) d\sigma_n(\mathbf{x})= \frac{1}{|C|} \sum_{\mathbf{x} \in C} p(\mathbf{x}) 
\end{equation}
($\sigma_n $ is the normalized surface measure) holds for all polynomials $p(\mathbf{x}) = p(x_1,x_2,\ldots,x_n)$ of total degree at most $\tau$. 
\end{definition}

Apparently, codes with small $|I(C)|$ which are spherical $\tau$-designs with large $\tau$ are good candidates for being well-distributed. In this paper we consider a class of codes on $\mathbb{S}^{31}$ which arise from even extremal unimodular lattices in $\mathbb{R}^{32}$. Such codes are spherical $7$-designs and have $|I(C)|=6$ as $0,-1 \in I(C)$ and the set $I(C) \setminus \{-1\}$ is symmetric about $0$. With suitably chosen $T$, these codes are $T$-avoiding universally optimal and enjoy optimality as maximal codes and minimal designs in their class. We provide corresponding statements in Sections 4--6.

The interest in $T$-avoiding spherical codes came apparently after
Gon{\c{c}}alves and Vedana~\cite{GV2025} proved that four known extremal even unimodular lattices in $\mathbb{R}^{48}$ have the maximal possible density among the lattices with certain forbidden distances. The proof in~\cite{GV2025} continues the research of  Viazovska~\cite{V} and Cohn--Kumar--Miller--Radchenko--Viazovska~\cite{CKMRV17,CKMRV22} who show the sphere packing optimality of $E_8$ and Leech lattices. 
In~\cite{BC2025}, $T$-avoiding $(48, 52 416 000,1/2)$-codes with $T=(-1/3,-1/6) \cup (1/6,1/3)$ were proved to be maximal as $T$-avoiding $1/2$-codes (with certain mild assumptions) and minimal as $T$-avoiding spherical 11-designs. The universal optimality of these codes has been studied in~\cite{BD2025}. The paper~\cite{BCD2025} considers $T$-avoiding codes in other dimensions.

\section{Preliminaries}

We apply a linear programming technique which is briefly explained in this section. 

For a given dimension $n$, we associate the sequence of Gegenbauer polynomials $\left(P_i^{(n)}\right)_{i=0}^{\infty}$ (see ~\cite{Szego1975}), that are orthogonal on $[-1,1]$ with weight $w(t) := (1-t^2)^{(n-3)/2}$, which we normalize by $P_i^{(n)}(1) = 1$. 
Given a code $C\subset \mathbb{S}^{n-1}$, the $i$-th moment of $C$, $i \geq 1$, is defined by
\[
M_i(C):=\sum_{\mathbf{x}, \mathbf{y} \in C} P_i^{(n)}( \mathbf{x \cdot y}).
\]
A key property of Gegenbauer polynomials is their positive definiteness~\cite{Sch42}, which means that $M_i (C) \geq 0$ for every code $C$ and positive integer $i$.

A code $C$ is a spherical $\tau$-design if and only if $M_i(C)=0$, $i=1,\dots,\tau$. Note also that for an antipodal $C$ (such that $C = -C$) it follows that $M_i(C) = 0$ for all odd $i$.

The linear programming technique we use is based on building a function $f$ with a proper sign on $[-1,1]\setminus T$ and proper coefficients in the Gegenbauer expansion
\begin{equation} \label{geg-exp}
f(t) = \sum_{i=0}^{\deg f} f_i P_i^{(n)}(t).
\end{equation}
The Fourier coefficients in~\eqref{geg-exp} can be computed in a standard way.

The following key identity (see~\cite{DGS,Lev92}) is based on a double-counting for $f(\mathbf{x \cdot y})$ over $C \times C$:
\begin{equation} \label{main-id}
f(1) \cdot |C|+\sum_{\mathbf{x},\mathbf{y} \in C, \mathbf{x} \neq \mathbf{y}} f( \mathbf{x \cdot y} )=
f_0 \cdot |C|^2 + \sum_{i=1}^{\deg{f}} f_iM_i(C),
\end{equation}
where $f(t)$ is a polynomial and $C \subset \mathbb{S}^{n-1}$ is a code. Additional assumptions on $f$, $T \subset [-1,1]$ and moments $M_i(C)$ give bounds on the cardinalities of $T$-avoiding codes, minimal $T$-avoiding designs and on $h$-energy of $C$.

For any point $\mathbf{x} \in C$ and inner product $t \in I(C)$, we denote by $A_t(\mathbf{x})$
the number of points of $C$ with inner product $t$ with $\mathbf{x}$. The system of nonnegative integers $(A_t(\mathbf{x}): t \in I(C))$ is called the \textit{distance distribution
of $C$ with respect to the point $\mathbf{x}$}. If the numbers $A_t(\mathbf{x})$ do not depend on the choice of $\mathbf{x}$ for any $t \in I(C)$, then the code $C$ is called {\it distance invariant}. If the code is distance invariant, we shall omit $\mathbf{x}$ in the notation $A_t(\mathbf{x})$. 
Then the distance distribution of a distance-invariant code $C$ is naturally defined as the system of positive integers
\[ 
F(C):= \{ A_{t_i} : t_i \in I(C) \cup \{1\} \}, 
\]
where $I(C)=\{t_1,\ldots,t_d\}$ and $t_1< \cdots < t_d < t_{d+1} := 1$. 
The distance distribution of a distance-invariant $\tau$-design with $d \leq \tau-1$ can be found from a Vandermonde-type system, see~\cite[Theorem 7.4]{DGS}. Furthermore, the definition of design implies the following quadrature formula for a polynomial $f$ of degree at most $\tau$ 
\begin{equation} \label{QF}
N f_0 = \sum_{i=1}^{d+1} A_{t_i} f(t_{i}).
\end{equation}
Note that the above-mentioned Vandermonde system can be obtained from~\eqref{QF} for $f = 1,t,\dots, t^d$. 

\section{Extremal even unimodular lattices in \texorpdfstring{$\mathbb{R}^{32}$}{R32} and their codes}

\subsection{Lattices} 

The codes we are interested in are found as the sets of minimum length vectors in {\it extremal even unimodular lattices} as defined below. We first provide general definitions and then focus on dimension 32. 

A lattice $L \in \mathbb{R}^n$ is called {\it even} if all norms (squared lengths) of vectors of $L$ are even. It is called \textit{unimodular} if it coincides with its dual lattice $L^{*}:=\{ \mathbf{x} \in \mathbb{R}^n : \mathbf{x \cdot y} \in \mathbb{Z} \ \forall \mathbf{y} \, \in L \}$. If $L$ is even unimodular, then 8 divides $n$. It is known that the minimum norm of an even unimodular lattice $L$ satisfies 
\begin{equation} \label{L-bound}
    \min \{ \mathbf{x \cdot x} : \mathbf{x} \in L, x \neq 0 \} \leq 2\left[\frac{n}{24}\right]+2. 
\end{equation}
The vectors of $L$ can be partitioned into \emph{layers}
\[ L_m:=\{ \mathbf{x} \in L: \mathbf{x \cdot x} =m \}. \]
An even unimodular lattice is called \textit{extremal} if the bound~\eqref{L-bound} is attained; in other words, in the extremal lattices the layers $L_m$ with $m<2[n/24]+2$ are empty. 

In $\mathbb{R}^{32}$ the extremal even unimodular lattices have not been classified yet. King~\cite{King} proved that there are more than $10^7$ non-isometric extremal even unimodular lattices without roots.
All these lattices attain the bound~\eqref{L-bound} and are extreme; i.e., they realize a local maximum of the density function~\cite{nebe2013boris}. 
Recently these lattices were examined in terms of optimality for the Gaussian potential function
\[
E(\alpha, L) = \sum_{\mathbf{x} \in L} e^{-\alpha (\mathbf{x \cdot x})},
\]
see Heimendahl--Marafioti--Thiemeyer--Vallentin--Zimmermann~\cite{heimendahl2023critical}. In particular, they show that for $\alpha = \pi$ such a lattice is a local maximum.

We consider the set $\mathcal C$ of spherical codes obtained by rescaling the first non-empty layer $L_4$ of an extremal even unimodular lattice $L$.

\subsection{The set \texorpdfstring{$\mathcal C$}{C} of codes on \texorpdfstring{$\mathbb{S}^{31}$}{S31} from minimum norm 4 vectors}

The first nontrivial layer $L_4$ of an extremal even unimodular lattice in $\mathbb{R}^{32}$ contains 146880 vectors. After rescaling to $\mathbb{S}^{31}$, we obtain a spherical code $C \in \mathcal{C}$ with cardinality $|C|=146880$, set of inner products 
\begin{equation}\label{InnerProductsSet}
    I(C) = \left \{ -1, - \frac{1}{2}, - \frac{1}{4}, 0,\frac{1}{4}, \frac{1}{2} \right\}.
\end{equation}
Furthermore, $C$ is a spherical $7\sfrac{1}{2}$-design; i.e., a 7-design with $M_{10}(C) = 0$. Since $|I(C)|=6<\tau=7$, $C$ is distance regular~\cite[Theorem 7.4]{DGS}. Its distance distribution is given by
\begin{equation}\label{DistanceDistribution}
    \mathcal{F}(C)=\{1,1240,31744,80910,31744,1240,1\}.
\end{equation}
The quadrature rule~\eqref{QF} takes the form
\begin{equation}\label{QF146}
  146880f_0 =  f(-1) + f(1) + 1240 \left ( f \left(-\frac{1}{2} \right) + f\left(\frac{1}{2} \right) \right) + 31744 \left  ( f\left(-\frac{1}{4}\right) + f\left(\frac{1}{4}\right)\right),  
\end{equation}
holding true for every polynomial $f$ of degree at most 7.

For certain specific choices of $T$, we shall prove that all these codes are simultaneously maximal spherical $T$-avoiding codes among the codes that are 3-designs (an additional  condition of such a type is inevitable, see Remark~\ref{Leech-obstruction}).
Also, these codes are minimal (tight) $T$-avoiding spherical 7-designs.

\subsection{Lattices in \texorpdfstring{$\mathbb{R}^{32}$}{R32} obtained from self-dual binary codes}

Of the more than $10^7$  extremal even unimodular lattices in dimension $32$, there are few that have been studied in detail~\cite{venkov1989ganzzahlige}.
Five of them can be called special since they can be obtained from the five (see Conway--Pless~\cite{CP} for the classification) extremal doubly-even binary self-dual codes of length 32, denoted by $CP_{i},$ $i=1,2,\ldots ,5$. As it follows from their self-duality and extremality, all these five codes have dimension 16 and minimum distance 8. 

Let $C$ be one of those five codes. We define a lattice $L(C)$
using construction B (see~\cite[Chapter 7]{CS}) followed
by an extension (doubling). One first maps $C$ into $\mathbb{R}^{32}$ by
\begin{equation*}
G(C) = \left\{ \mathbf{x=}(x_{1},\ldots x_{32}) \in \mathbb{Z}
^{32}~|~x_{1}+\cdots +x_{32}\equiv 0 \pmod{4},~\mathbf{x} \pmod{2}
\in C\right\} .
\end{equation*}
Then a ``doubling'' is performed via 
\begin{equation*}
L(C)=\frac{1}{\sqrt{2}} G(C)\cup \frac{1}{\sqrt{2}}\left( \left (\frac{1}{2},\frac{1}{2},\cdots ,\frac{1}{2} \right) + G(C)\right).
\end{equation*}
The resulting lattice is extremal, even, and unimodular. We are mostly interested in $L(C)_{4}$, the first non-empty layer of $L(C)$. 

In search of classifying patterns, Venkov~\cite{Ven82} introduced and investigated the parameter 
\begin{equation*}
D(L) = \{ e_{2,2}(\mathbf{x},\mathbf{z}) = \left\vert \left\{ 
\mathbf{y} \in L(C)_{4}~|~\mathbf{x} \cdot \mathbf{y}=2,~\mathbf{z}\cdot \mathbf{y}=2\right\} \right\vert \, , \, \mathbf{x},\mathbf{z} \in L(C)_{4} \mbox{ with } (\mathbf{x},\mathbf{z}) = 0\}.
\end{equation*}
He showed that $D(L)$ consists of even integers between $0$ and $60$. 
We see that $D(L(CP_i))$ contains the maximum value $60$ as follows.
The vectors 
\[ \mathbf{x}=\frac{1}{\sqrt{2}}(0,0,\ldots ,0,2,2), \ \ 
\mathbf{z}=\frac{1}{\sqrt{2}}(0,0,\ldots ,0,-2,2) \] 
(note that $\mathbf{x\cdot z=0}$) and 
\[ \mathbf{y}_{1} = \frac{1}{\sqrt{2}}(2,0,\ldots ,0,2),\,\mathbf{y}_{2}=\frac{1}{\sqrt{2}}(0,2,\ldots ,0,0,2),\ldots ,\,\mathbf{y}_{30}=\frac{1}{\sqrt{2}}(0,0,\ldots ,2,0,2), \]
\[ \mathbf{y}_{31} = \frac{1}{\sqrt{2}}(-2,0,\ldots ,0,2),\,\mathbf{y}_{32}=\frac{1}{\sqrt{2}}(0,-2,\ldots ,0,0,2),\ldots ,\,\mathbf{y}_{60}=
\frac{1}{\sqrt{2}}(0,0,\ldots ,-2,0,2) \]
belong to the norm 4 layer of $L(C)$ and $\mathbf{x\cdot y}_{i}=\mathbf{z\cdot y}_{i}=2$ for $i=1,2,\ldots ,60.$ 

Venkov showed that every lattice $L$ with $60 \in D(L)$ can be constructed as above, and so such $L$ coincides with one of the $L(CP_i)$.
In order to show that the five lattices $L(CP_{i})$ are not isometric Koch and Venkov~\cite{venkov1989ganzzahlige} classified $D(L(CP_i))$ for $1 \leq i \leq 5$ and noted that these sets are distinct.

We confirm this result using the computational algebra system Magma~\cite{Mag}, by showing that the five lattices have automorphism groups of different orders, namely:
\begin{eqnarray*}
975175680 &=&2^{21}\cdot 3\cdot 5\cdot 31, \\
48126558103142400 &=&2^{31}\cdot 3^{5}\cdot 5^{2}\cdot 7\cdot 17\cdot 31, \\
91321742131200 &=&2^{31}\cdot 3^{5}\cdot 5^{2}\cdot 7, \\
16911433728 &=&2^{28}\cdot 3^{2}\cdot 7, \\
1509949440 &=&2^{25}\cdot 3^{2}\cdot 5.
\end{eqnarray*}
Currently in the Magma database of lattices, there are 19 non-isometric
extremal even unimodular lattices of dimension 32. Among them, the
lattices $L(CP_{i}),$ $i=2,3,4,5,$\ have the largest automorphism groups. 
There is only one lattice with an automorphism group order greater than 
$2^{21}\cdot 3\cdot 5\cdot 31$ and that number is $1244160000=2^{13}\cdot
3^{5}\cdot 5^{4}.$ The least automorphism group order is $12288=2^{12}\cdot
3.$

\section{Maximal \texorpdfstring{$T$}{T}-avoiding (1/2)-codes in 32 dimensions}

In this section we prove that every code $C \in \mathcal{C}$ is a maximal $T$-avoiding $(1/2)$-code for $T=(0,1/4)$ under the additional assumption that the code is at least a 3-design. As said above, all such codes have cardinality 146880 and are, in fact, spherical $7\sfrac{1}{2}$-designs. 

\begin{remark} \label{Leech-obstruction}
One may be tempted to favor the embedding of the kissing configuration of the Leech lattice as a maximal $T$-avoiding $(1/2)$-code in 32 dimensions for some $T$. However, an easy calculation shows that the second moment of that embedding is positive; i.e., it is not even a 2-design. This prevents a contradiction when maximality is proven via polynomials in~\eqref{main-id} with negative coefficient $f_2$ for codes which are at least 2-designs. Therefore, we can possibly have maximal codes with less than 196560 points if we additionally require that our codes are at least spherical 2-designs\footnote{In fact, our proofs will require strength 3.}. 
\end{remark}

\begin{thm} \label{146880-codes-t1}
Let $C \subset \mathbb{S}^{31}$ be a $(0,1/4)$-avoiding $(1/2)$-code that is (at least) a spherical 3-design. Then $|C| \leq 146880$. 
This bound is attained by every code $C \in \mathcal{C}$.
\end{thm} 

\begin{proof}
Assume that $C \subset \mathbb{S}^{31}$ is a $(0,1/4)$-avoiding $(1/2)$-code and a spherical 3-design. Consider~\eqref{main-id} with $C$ and the degree 10 polynomial
\[ f(t)= (t+1)\left(t + \frac{1}{2}\right)^2 \left(t+\frac{1}{4}\right)^2 t\left(t-\frac{1}{4}\right) \left(t-\frac{1}{2}\right)^3. \]
Note that $f(1)=675/1024$ and $f_0=5/1114112$ in the Gegenbauer expansion~\eqref{geg-exp} of $f$. 

Since $f(t) \leq 0$ for every $t \in [-1,1/2] \setminus (0,1/4)$, the left-hand side of~\eqref{main-id} does not exceed $f(1)|C|=675|C|/1024$. For the right-hand side we first see that since $M_i(C)=0$ for $i=1,2,3$, the signs of the coefficients $f_i$, $i=1,2,3$, do not matter (in fact, we have $f_2<0$ and $f_3<0$; see the Leech lattice obstruction Remark~\ref{Leech-obstruction}).  Next, we observe that $f_i>0$ for $i \in \{4,\ldots,10\} \setminus \{8\}$, while $f_8=0$. We conclude that the right-hand side of~\eqref{main-id} is at least $f_0|C|^2=5|C|^2/1114112$. This gives 
\[ |C| \leq \frac{f(1)}{f_0} = 146880. \]

It follows from the above argument that any attaining code $C$ must be a spherical 7-design with vanishing tenth moment, i.e. a spherical $7\sfrac{1}{2}$-design. Moreover, the inner products of $C$ must belong to $\{-1, \pm 1/2, \pm 1/4, 0 \}$. Every code $C \in \mathcal{C}$
possesses these properties and, therefore, attains the bound. This completes the proof. 
\end{proof}

The polynomial for Theorem~\ref{146880-codes-t1} was constructed based on the properties of the attaining codes. Indeed, the fact that the codes are antipodal spherical 7-designs with $M_{10}=0$ prompts us to use a degree-10 polynomial with $f_8=0$. Further, the necessary vanishing at the inner products and the endpoints of the interval $T$ determine the roots of the polynomial to be as in the set $I(C)$ from~\eqref{InnerProductsSet} with the corresponding multiplicities. Lucky or not, the resulting polynomial is good as shown in the proof. 

We conclude this section with the explicit Gegenbauer expansion of the polynomial $f$ used in the above proof. It is 
\[ f(t)=\sum_{i=0}^{10} f_iP_i^{(32)}(t), \]
with
\[ f_0=\frac{5}{1114112}, \ f_1=\frac{65}{992256}, \ f_2=-\frac{31}{196608},
\ f_3=-\frac{93}{165376}, \ f_4=\frac{217}{417792}, \] \[ f_5=\frac{899}{58368}, \ f_6=\frac{2387}{188416},\ f_7=\frac{20119}{894976}, \ f_8=0, \ f_9=\frac{3441}{11776}, \ f_{10}=\frac{14911}{47104}. \]

\section{Tight \texorpdfstring{$T$}{T}-avoiding spherical 7-designs in 32 dimensions}

In this section we show how minimality results for $T$-avoiding spherical designs on $\mathbb{S}^{31}$ with some $T$ can be derived. Such designs were called {\it tight $T$-avoiding} in~\cite{BCD2025}. We shall prove that every code $C \in \mathcal{C}$ is a tight $T$-avoiding spherical $7$-design with $T = (-1/4,0) \cup (1/4,1/2)$. 

\begin{thm} \label{146880-des-t1}
Let $C \subset \mathbb{S}^{31}$ be a $T$-avoiding spherical 7-design, where 
$T=(-1/4,0) \cup (1/4,1/2)$. Then $|C| \geq 146880$. 
This bound is attained by every code $C \in \mathcal{C}$.
\end{thm} 

\begin{proof}
We apply~\eqref{main-id} with the degree 7 polynomial
\[ f(t)= t(t+1)\left(t + \frac{1}{2}\right)^2 \left(t+\frac{1}{4}\right) \left(t-\frac{1}{4}\right) \left(t-\frac{1}{2}\right). \]
It is easy to see that $f(t) \geq 0$ for every $t \in [-1,1] \setminus T$ for the left-hand side and that the right-hand side is equal to $f_0|C|^2$. Since $f(1)=135/64$ and $f_0=1/69632$, we obtain $|C| \geq f(1)/f_0=146880$. 
\end{proof}

Again, the properties of the attaining designs suggest the right polynomial for use in~\eqref{main-id}. It has degree 7 to provide $\sum_{i=1}^7 f_i M_i(C)=0$ in the right-hand side and zeros in $I(C)$ from~\eqref{InnerProductsSet} with appropriate multiplicities to make the sum in the left-hand side nonnegative.

\section{Universally optimal \texorpdfstring{$T$}{T}-avoiding codes in 32 dimensions}

Let $L$ be any fixed even unimodular extremal lattice in $\mathbb{R}^{32}$ and $L_4$ be the collection of $146880$ vectors of minimal norm. Among several possible choices of $T$ we consider the symmetric case, where 
$$T:=\left(-\frac12,-\frac14\right)\cup \left(\frac14,\frac12\right).$$ We denote the class of $T$-avoiding codes on $\mathbb{S}^{31}$ by 
\begin{equation}\label{C_Tsym}
    \mathcal{C}_{T} := \{C \subset\mathbb{S}^{31}: |C|=146880,\quad I(C)\cap T= \varnothing \} .
\end{equation}
The following theorem establishes that for any absolutely monotone potential function $h$ a universal error bound on the $h$-energy of codes in the class $\mathcal{C}_{T}$ holds.

\begin{thm} \label{146880-en-t1} Let $h$ be absolutely monotone with $h^{(8)}>0$ in $(-1,1)$ and $C \in \mathcal{C}_{T}$. Then 
\begin{equation}\label{E146880} 
\begin{split}
E_h(C) \geq & \ 146880\left[ h(-1)+1240h\left(-\frac{1}{2}\right)+31744h\left(-\frac{1}{4}\right) + 80910h\left(0\right) \right.\\
& \qquad \qquad +\left.31744h\left( \frac{1}{4}\right)+1240h\left(\frac{1}{2}\right) \right].
\end{split}
\end{equation}
This bound is attained by every code $C \in \mathcal{C}$.
\end{thm}

\begin{proof}
First, we note that every code $C \in \mathcal{C}$ has the same energy as given by the right-hand side of~\eqref{E146880}. Indeed, any such code has the same inner products and distance distribution as explained in Section 3.2 and the calculation of its $h$-energy from~\eqref{h-energy} with the quadrature rule \eqref{QF146} gives the right-hand side of~\eqref{E146880}.

So, our goal is to show that~\eqref{E146880} holds for any $T$-avoiding code $C$. Consider the multiset 
$$\mathcal{I}:=\{t_1,\dots,t_8\}=\left\{-1,-1,-\frac{1}{2},-\frac{1}{4},0,0,\frac{1}{4}, \frac{1}{2}\right\}.$$
Denote by $H_7(t;h,\mathcal{I})$ the polynomial, which interpolates the potential function $h$ at the points of $\mathcal{I}$ (a repeated node indicates derivative interpolation). Newton's interpolation formula implies that
\begin{equation}\label{NewtonsFormula}
    H_7(t;h,\mathcal{I})= h(t_1)+\sum_{i=1}^7 h[t_1,\dots,t_{i+1}](t-t_1)\dots(t-t_i),
\end{equation}
where $h[t_1,\dots,t_{i+1}]$,  $i=1,\dots,7$, are the corresponding divided differences and 
$$P_i(t):=(t-t_1)\dots(t-t_i), \quad i=1,\dots,7$$ are the associated partial products. The absolute monotonicity of $h$ implies that all divided differences are non-negative. 

The positive definiteness of the partial products $P_i(t)$ can be verified directly, with which we now proceed. 
Since $t+c$ has positive Gegenbauer coefficients for any $c\geq 0$ and positive definiteness is preserved under multiplication of polynomials (this is the so-called Krein condition in~\cite{Lev98}), we conclude that the partial products $P_i$, $i=1,\dots,6$, are positive definite. The last partial product 
\[ P_7(t)=\left(t + 1\right)^2 \left(t+\frac{1}{2}\right) \left(t+\frac{1}{4}\right)t^2 \left(t-\frac{1}{4}\right)=\sum_{i=0}^7 f_iP_i^{(32)}(t), \]
with 
\[f_0=\frac{97}{104448}, \ f_1=\frac{619}{41344}, \ f_2=\frac{12245}{116736}, \ f_3=\frac{2139}{5168}, f_4=\frac{13981}{13056}, \]
\[f_5=\frac{4433}{2432}, \ f_6=\frac{11935}{7296}, \ f_7= \frac{341}{608}, \]
is explicitly seen to be positive definite. 
We next utilize the Hermite error formula, which yields (recall that $T=(-1/2,-1/4)\cup (1/4,1/2)$)
\[h(t)-H_7 (t;h,\mathcal{I})=\frac{h^{(8)}(\xi)}{8!}(t+1)^2\left(t+\frac{1}{2}\right)\left(t+\frac{1}{4}\right)t^2\left(t-\frac{1}{4}\right)\left(t-\frac{1}{2}\right)\geq 0\]
for every $t\in [-1,1] \setminus T$.

To complete the argument, we combine the $h$- and $f$-energies using~\eqref{main-id} with $C$, $H_7(t;h,\mathcal{I})$, and the above conclusions. In the left-hand side of \eqref{main-id} we estimate
\[ E_h(C) =\sum_{\mathbf{x},\mathbf{y} \in C, \mathbf{x} \neq \mathbf{y}} h(\mathbf{x \cdot y}) \geq \sum_{\mathbf{x},\mathbf{y} \in C, \mathbf{x} \neq \mathbf{y}} f(\mathbf{x \cdot y}) \]
since $h(t) \geq H_7 (t;h,\mathcal{I})$ for all inner products of $C$. The right-hand side is at least $(H_7)_0|C|^2$ since all Gegenbauer coefficients of $H_7 (t;h,\mathcal{I})$ are non-negative. Therefore, we have
\[ E_h (C)\geq 146880^2 \left( (H_7)_0-\frac{H_7(1)}{146880}\right). \]
It remains to see that the right-hand side of the last inequality is equal to the right-hand side of the bound~\eqref{E146880}. 
This fact follows from the quadrature formula \eqref{QF} written for any of the 7-designs in~$\mathcal {C}$.

If equality holds for some code $C$, then all moments $M_i(C)$, $i=1,\dots,7$, vanish, so $C$ has to be a $7$-design. Moreover, the interpolation conditions imply that the inner products of $C$ belong to the set $\{0,-1, \pm 1/2,\pm 1/4\}$ (see~\eqref{InnerProductsSet}). From~\cite[Theorem 7.4]{DGS} such a code is distance-invariant. Utilizing the quadrature formula \eqref{QF} associated with the design property, we conclude that such a code has the same distance distribution as~\eqref{DistanceDistribution}, which concludes the proof.
\end{proof}

\bigskip
{\bf Funding.} The research of the first two authors is supported, in part, by Bulgarian NSF grant KP-06-Austria/8-2025. 
The research of the third author is supported, in part, by the Lilly Endowment.

\end{document}